\documentclass[12pt,reqno]{amsart}
\usepackage{amssymb,a4wide,eucal,amscd,yhmath,amsgen,amsopn}
\usepackage{fancyhdr}
\usepackage{amsmath,hyperref,xcolor}
\usepackage{mathrsfs}
\usepackage[all,cmtip]{xy}
\SilentMatrices
\SelectTips{eu}{}
\newtheorem{theorem}{Theorem}[section]

\newtheorem{lemma}[theorem]{Lemma}
\newtheorem{Not}[theorem]{Notation}
\theoremstyle{rem}

\theoremstyle{definition}
\newtheorem{definition}[theorem]{Definition}
\theoremstyle{construct}

\theoremstyle{examp}

\newtheorem{Fact}[theorem]{Fact}
\newcommand\projective\mathbf
\newcommand\PP{\projective P}

\newcommand\OO{\mathcal O}

\newcommand\ZZ{\mathbb Z}

\newcommand\onto\twoheadrightarrow
\newcommand\lra\longrightarrow
\newcommand\dar\downarrow
\DeclareMathOperator{\pic}{Pic}
\DeclareMathOperator{\im}{im}
\DeclareMathOperator{\cok}{coker}
\DeclareMathOperator{\rk}{rank}
\DeclareMathOperator{\Hom}{Hom}

\begin{document}

\title{Simple Cohomology bundles on multiprojective spaces}
\author{Damian M Maingi}
\date{December,2025}

\address{Department of Mathematics\\Catholic University of Eastern Africa\\P.O Box 62157, 00200 Nairobi}
\email{dmaingi@cuea.edu}

\maketitle

\begin{abstract}
We prove stability of the kernel bundle and prove that the cohomology bundle is simple for vector bundles associated to monads on
$X = (\PP^{n_1})^2\times\cdots\times(\PP^{n_s})^2$ for an ample line bundle $\mathscr{L}=\mathcal{O}_X(\alpha_1,\alpha_1,\cdots,\alpha_s,\alpha_s)$.
\end{abstract}

\section{Introduction}

\noindent In this work we give generalizations for previous results by several authors see \cite{1,2,3,6,9}. 
In the paper we generalize the results by Maingi in \cite{6} where the ambient space is $X = (\PP^{n_1})^2\times\cdots\times(\PP^{n_s})^2$ and
\cite{9} where the polarisation is $\mathscr{L}=\mathcal{O}_X(\alpha_1,\alpha_1,\cdots,\alpha_s,\alpha_s)$.
We build upon results by Maingi \cite{4,5,6,7,8,9} therefore the definitions, notation, the methods applied are quite similar and the trend follows the paper
\cite{6}. We study the kernel vector bundles and the cohomology bundle associated to monads on $X$ and prove their stability and simplicity respectively.\\

\noindent The main results in this paper are:

\begin{theorem}
Let $\alpha_1,\cdots,\alpha_s$ and $k$ be nonnegative integers. Then there exists a monad on $X = (\PP^{n_1})^2\times(\PP^{n_2})^2\times\cdots\times(\PP^{n_s})^2$ of the form;
\[\begin{CD}0\rightarrow{\OO_X(-\alpha_1,-\alpha_1,\cdots,-\alpha_s,-\alpha_s)^{\oplus k}} @>>^{f}>{\mathscr{G}_{\alpha_1}\oplus\cdots\oplus\mathscr{G}_{\alpha_s}}@>>^{g}>\OO_X(\alpha_1,\alpha_1,\cdots,\alpha_s,\alpha_s)^{\oplus k}\rightarrow0\end{CD}\]
where 
\begin{align*}\mathscr{G}_{\alpha_1}:=\OO_X(-\alpha_1,0,0,\cdots,0)^{\oplus n_1+\oplus k}\oplus\OO_X(0,-\alpha_1,0,0,\cdots,0)^{\oplus n_1+\oplus k}\\
\mathscr{G}_{\alpha_2}:=\OO_X(0,0,-\alpha_2,\cdots,0)^{\oplus n_2+\oplus k}\oplus\OO_X(0,0,0,-\alpha_2,\cdots,0)^{\oplus n_2+\oplus k}\\
\cdots\cdots\cdots\cdots\cdots\cdots\cdots\cdots\cdots\cdots\cdots\cdots\cdots\cdots\cdots\cdots\cdots\cdots\cdots\\
\mathscr{G}_{\alpha_s}:=\OO_X(0,0,\cdots,0,-\alpha_s,0)^{\oplus n_s+\oplus k}\oplus\OO_X(0,0,\cdots,0,-\alpha_s)^{\oplus n_s+\oplus k}
\end{align*}
with the properties

\begin{enumerate}
 \item The kernel of $g$, $\ker(g)$ is stable and
 \item The cohomology bundle $E=\ker g/\im f$ is simple.
\end{enumerate}
\end{theorem}

\begin{Not}

\noindent The ambient space is the Cartesian product $X = (\PP^{n_1})^2\times\cdots\times(\PP^{n_s})^2$.
\\
Suppose $p_{\alpha_{11}}$ and $p_{\alpha_{12}}$ are natural projections  from $X$ onto $\PP^{n_1}$,\\
$p_{\alpha_{21}}$ and $p_{\alpha_{22}}$ are natural projections  from $X$ onto $\PP^{n_2}$, \\
$\cdots\cdots\cdots\cdots\cdots\cdots\cdots\cdots\cdots\cdots\cdots\cdots\cdots\cdots\cdots\cdots\cdots\cdots\cdots$ and \\
\\
$p_{\alpha_{s1}}$ and $p_{\alpha_{s2}}$ are natural projections  from $X$ onto $\PP^{n_s}$.\\
\\
We shall denote by:
$g_{\alpha_{11}}$ the generator of $\pic(X)$ corresponding to $p_{\alpha_{11}}^*\OO_{\PP^{n_1}}(1)$,\\
$g_{\alpha_{12}}$ the generator of $\pic(X)$ corresponding to $p_{\alpha_{12}}^*\OO_{\PP^{n_1}}(1)$,\\
$g_{\alpha_{21}}$ the generator of $\pic(X)$ corresponding to $p_{\alpha_{21}}^*\OO_{\PP^{n_2}}(1)$,\\
$g_{\alpha_{22}}$ the generator of $\pic(X)$ corresponding to $p_{\alpha_{22}}^*\OO_{\PP^{n_2}}(1)$,\\
$\cdots\cdots\cdots\cdots\cdots\cdots\cdots\cdots\cdots\cdots\cdots\cdots\cdots\cdots\cdots\cdots\cdots\cdots\cdots$ and \\
$g_{\alpha_{s1}}$ the generator of $\pic(X)$ corresponding to $p_{\alpha_{21}}^*\OO_{\PP^{n_s}}(1)$,\\
$g_{\alpha_{s2}}$ the generator of $\pic(X)$ corresponding to $p_{\alpha_{22}}^*\OO_{\PP^{n_s}}(1)$,\\
and so the Picard group of $X$ is $\pic(X)= \left\langle g_{\alpha_{11}},g_{\alpha_{12}},g_{\alpha_{21}},g_{\alpha_{22}},\ldots,g_{\alpha_{s1}},g_{\alpha_{s2}}\right\rangle$.
\\
We also denote by \\
$\OO_X(g_{\alpha_{11}},g_{\alpha_{12}},\ldots,g_{\alpha_{s1}},g_{\alpha_{s2}}):= 
p_{\alpha_{11}}^*\OO_{\PP^{n_1}}(g_{\alpha_{11}})\otimes p_{\alpha_{12}}^*\OO_{\PP^{n_1}}(g_{\alpha_{12}})\otimes\cdots\otimes p_{\alpha_{s1}}^*\OO_{\PP^{n_s}}(g_{\alpha_{s1}})\otimes p_{\alpha_{s1}}^*\OO_{\PP^{n_s}}(g_{\alpha_{s2}})$.
\\
Suppose $h_{\alpha_{11}}$ and $h_{\alpha_{12}}$ are hyperplanes in $\PP^{n_1}$,\\
$h_{\alpha_{21}}$ and $h_{\alpha_{22}}$ are hyperplanes in $\PP^{n_2}$,\\
$\cdots\cdots\cdots\cdots\cdots\cdots\cdots\cdots\cdots\cdots\cdots\cdots\cdots\cdots\cdots\cdots\cdots\cdots\cdots$\\
$h_{\alpha_{s1}}$ and $h_{\alpha_{s2}}$ are hyperplanes in $\PP^{n_s}$, with the intersection product induced by\\
$h_{\alpha_{11}}^{n_1} = h_{\alpha_{12}}^{n_1} = h_{\alpha_{21}}^{n_2} = h_{\alpha_{22}}^{n_2} = \cdots = h_{\alpha_{s1}}^{n_s} = h_{\alpha_{s2}}^{n_s} = 1$ and \\
$h_{\alpha_{11}}^{n_1+1} = h_{\alpha_{12}}^{n_1+1} = h_{\alpha_{21}}^{n_2+1} = h_{\alpha_{22}}^{n_2+1} = \cdots = h_{\alpha_{s1}}^{n_s+1} = h_{\alpha_{s2}}^{n_s+1} = 0$
$g_{in}^{n+1} = g_{im}^{m+1} = g_{il}^{l+1} = 0$.
\\
For any line bundle $\mathscr{L} = \OO_X(\alpha_1, \alpha_1,\cdots,\alpha_s,\alpha_s)$ on $X$ and a vector bundle $E$, we write 
$E(\alpha_1, \alpha_1,\cdots,\alpha_s,\alpha_s) = E\otimes\OO_X(\alpha_1, \alpha_1,\cdots,\alpha_s,\alpha_s)$.\\
\\
For any line bundle $\mathscr{L} = \OO_X(\alpha_1, \cdots,\alpha_s)$ on $X$ and a vector bundle $E$, we write 
$E(\alpha_1, \cdots,\alpha_s) = E\otimes\OO_X(\alpha_1, \cdots,\alpha_s)$ .\\

\noindent The normalization of $E$ on $X$ with respect to $\mathscr{L}$ is defined as follows:\\
Set $d=\deg_{\mathscr{L}}(\OO_X(1,0,\cdots,0))$, since $\deg_{\mathscr{L}}(E(-k_E,0,\cdots,0))=\deg_{\mathscr{L}}(E)-nk\cdot \rk(E)$,\\ 
there is a unique integer $k_E:=\lceil\nu_\mathscr{L}(E)/d\rceil$ such that \\
$1 - d.\rk(E)\leq \deg_\mathscr{L}(E(-k_E,0,\cdots,0))\leq0$. \\
The twisted bundle $E_{{\mathscr{L}}-norm}:= E(-k_E,0,\cdots,0)$ is called the $\mathscr{L}$-normalization of $E$.\\
Finally we define the linear functional $\delta_{\mathscr{L}}$ on  $\mathbb{Z}^{2s}$ as \\
\\
$\delta_{\mathscr{L}}(p_{\alpha_{11}},p_{\alpha_{12}},\dots,p_{\alpha_{s1}},p_{\alpha_{s2}}):= \deg_{\mathscr{L}}\OO_{X}(p_{\alpha_{11}},p_{\alpha_{12}},\dots,p_{\alpha_{s1}},p_{\alpha_{s2}})$.\\

\end{Not}

\section{Preliminaries}

\noindent In this section we define and give notation in order to set up for the main results.
Most of the definitions are from chapter two of the book by Okonek, Schneider and Spindler \cite{12}.
In this paper we will work over an algebraically closed field of characteristic zero.

\begin{definition}
Let $X$ be a nonsingular projective variety. 
\begin{enumerate}
\renewcommand{\theenumi}{\alph{enumi}}
 \item A {\it{monad}} on $X$ is a complex of vector bundles:
\[\xymatrix{0\ar[r] & M_0 \ar[r]^{\alpha} & M_1 \ar[r]^{\beta} & M_2 \ar[r] & 0}\]
exact at $M_0$ and at $M_2$ i.e. $\alpha$ is injective and $\beta$ surjective.
\item The image of $\alpha$ is a subbundle of $B$ and the bundle $E = \ker(\beta)/\im (\alpha)$ and is called the cohomology bundle of the monad.
\end{enumerate}
\end{definition}

\begin{definition}
Let $X$ be a nonsingular projective variety, let $\mathscr{L}$ be a very ample line sheaf, and $V,W,U$ be finite dimensional $k$-vector spaces.
A linear monad on $X$ is a complex of sheaves,
\[ M_\bullet:
\xymatrix
{
0\ar[r] & V\otimes {\mathscr{L}}^{-1} \ar[r]^{A} & W\otimes \OO_X \ar[r]^{B} & U\otimes \mathscr{L} \ar[r] & 0
}
\]
where $A\in \Hom(V,W)\otimes H^0 \mathscr{L}$ is injective and $B\in \Hom(W,U)\otimes H^0 \mathscr{L}$ is surjective.\\
The existence of the monad $M_\bullet$ is equivalent to: $A$ and $B$ being of maximal rank and $BA$ being the zero matrix.
\end{definition}

\begin{definition}
Let $X$ be a non-singular irreducible projective variety of dimension $d$ and let $\mathscr{L}$ be an ample line bundle on $X$. For a 
torsion-free sheaf $F$ on $X$ we define
\begin{enumerate}
\renewcommand{\theenumi}{\alph{enumi}}
 \item the degree of $F$ relative to $\mathscr{L}$ as $\deg_{\mathscr{L}}F:= c_1(F)\cdot \mathscr{L}^{d-1}$, where $c_1(F)$ is the first Chern class of $F$
 \item the slope of $F$ as $\nu_{\mathscr{L}}(F):= \frac{\deg_{\mathscr{L}}F}{rk(F)}$.
  \item a torsion-free sheaf on $E$ is $\mathscr{L}$-stable if every subsheaf $F\hookrightarrow E$ satisfies $\mu_{\mathscr{L}}(F)<\mu_{\mathscr{L}}(E)$
  \end{enumerate}
\end{definition}

\subsection{Hoppe's Criterion over polycyclic varieties.}
Suppose that the Picard group $\pic(X) \simeq \ZZ^l$ where $l\geq2$ is an integer then $X$ is a polycyclic variety.
Given a divisor $B$ on $X$ we define $\delta_{\mathscr{L}}(B):= \deg_{\mathscr{L}}\OO_{X}(B)$.
Then one has the following stability criterion ({\cite{3}, Theorem 3}):

\begin{theorem}[Generalized Hoppe Criterion]
 Let $G\rightarrow X$ be a holomorphic vector bundle of rank $r\geq2$ over a polycyclic variety $X$ equipped with a polarisation 
 $\mathscr{L}$.
 \\
 If \[H^0(X,(\wedge^sG)\otimes\OO_X(B))=0\] 
 for all $B\in\pic(X)$ and $s\in\{1,\ldots,r-1\}$ such that
 $\begin{CD}\displaystyle{\delta_{\mathscr{L}}(B)<-s\nu_{\mathscr{L}}(G)}\end{CD}$ then $G$ is stable and if
 $\begin{CD}\displaystyle{\delta_{\mathscr{L}}(B)\leq-s\nu_{\mathscr{L}}(G)}\end{CD}$ then $G$ is semi-stable.\\
\\
 Conversely if then $G$ is (semi-)stable then  \[H^0(X,G\otimes\OO_X(B))=0\]
 for all $B\in\pic(X)$ such that $\left(\delta_{\mathscr{L}}(B)\leq\right)$ $\delta_{\mathscr{L}}(B)<-\nu_{\mathscr{L}}(G)$.
\end{theorem}

\noindent The following lemma is actually a corollary of Theorem 2.4 above, a special case of the generalized Hoppe criterion on stability.

\begin{lemma}
Let $X$ be a polycyclic variety with Picard number $n$, let $\mathscr{L}$ be an ample line bundle and
let E be a rank $r>1 $ vector bundle over $X$.
If $H^0(X,(\bigwedge^q E)_{{\mathscr{L}}-norm}(p_1,\cdots,p_{n})) = 0$ for $1\leq q \leq r-1$ and every 
$(p_1,\cdots,p_{n})\in \mathbb{Z}^{n}$ such that $\delta_{\mathcal L}(B)\leq0$, where $B:={\mathcal O}_X(p_1,..., p_n)$
then E is $\mathscr{L}$-stable.
\end{lemma}

\begin{Fact}
Let $0\rightarrow E \rightarrow F \rightarrow G\rightarrow0$ be an exact sequence of vector bundles. \\
Then we have the following exact sequences involving exterior and symmetric powers:\\
\begin{enumerate}
 \item $0\lra\bigwedge^q E \lra\bigwedge^q F \lra\bigwedge^{q-1} F\otimes G\lra\cdots \lra F\otimes S^{q-1}G \lra S^{q}G\lra0$\\
 \item $0\lra S^{q}E \lra S^{q-1}E\otimes F \lra\cdots \lra E\otimes\bigwedge^{q-1}F\lra\bigwedge^q F \lra\bigwedge^q G\lra 0$\\
\end{enumerate}
\end{Fact}

\begin{theorem}[\cite{13}, Theorem 4.1, page 131]
 Let $n\geq1$ be an integer  and $d$ be an integer. We denote by $S_d$ the space of homogeneous polynomials of degree $d$ in 
 $n+1$ variables (conventionally if $d<0$ then $S_d=0$). Then the following statements are true:
 \begin{enumerate}
 \renewcommand{\theenumi}{\alph{enumi}}
  \item $H^0(\PP^n,\OO_{\PP^n}(d))=S_d$ for all $d$.
  \item $H^i(\PP^n,\OO_{\PP^n}(d))=0$ for $0<i<n$ and for all $d$.
  \item The vector space $H^n(\PP^n,\OO_{\PP^n}(d))$ is isomorphic to the dual of the vector space $H^0(\PP^n,\OO_{\PP^n}(-d-n-1))$.
 \end{enumerate}
\end{theorem}

\begin{lemma}[\cite{9}, Lemma 2.11]
Let $X=\PP^{a_1}\times\cdots\times\PP^{a_n}$, $0\leq p< \dim(X) -1$ and $k$ be a positive integer.
If  $\displaystyle{\sum_{i=1}^np_i<}0$ then $h^p(X,\OO_X (p_1,\cdots,p_{n})^{\oplus k}) = 0$ . 
\end{lemma}

\begin{lemma}
Let $A$ and $B$ be vector bundles canonically pulled back from $A'$ on $\PP^n$ and $B'$ on $\PP^m$ then\\
$\displaystyle{H^q(\bigwedge^s(A\oplus B))=
\sum_{k_1+\cdots+k_s=q}\big\{\bigoplus_{i=1}^{s}(\sum_{j=0}^s\sum_{m=0}^{k_i}H^m(\wedge^j(A))\otimes(H^{k_i-m}(\wedge^{s-j}(B)))) \big\}}$.
\end{lemma}

\noindent The existence of monads on projective spaces was established by Fl\o{}ystad in \cite{2}, Main Theorem and generalized for a larger set of
projective varieties by Marchesi et al\cite{11}.

\begin{lemma}[\cite{11}, Theorem 2.4] Let $N\geq1$. There exists monads on $\PP^{N}$ whose maps are matrices of linear forms,
\[
\begin{CD}
0@>>>{\OO_{\PP^{N}}(-1)^{\oplus a}} @>>^{f}>{\OO^{\oplus b}_{\PP^{N}}} @>>^{g}>{\OO_{\PP^{N}}(1)^{\oplus c}} @>>>0\\
\end{CD}
\]
if and only if one of the following conditions holds\\
$(1) b\geq a+c$ and $b\geq 2c+N-1$ \\
$(2)b\geq a+c+N$.\\
If so there actually exists a monad with the map $f$ degenerating in expected codimension $b-a-c+1$.\\
If the cohomology of the monad is a vector bundle of rank less than $N$ then $N=2l+1$ is odd and the monad has the form
\[
\begin{CD}
0@>>>{\OO_{\PP^{2l+1}}(-1)^{\oplus a}} @>>^{f}>{\OO^{\oplus b}_{\PP^{2l+1}}} @>>^{g}>{\OO_{\PP^{2l+1}}(1)^{\oplus c}} @>>>0\\
\end{CD}
\] conversely for every $c,l\geq0$ there exists a monad as above whose cohomology is a vector bundle.
\end{lemma}

\begin{definition}
A vector bundle $E$ on $X$ is said to be simple if its only endomorphisms are the constants i.e. Hom$(E,E)=k$ which is equivalent
to $h^0(X,E\otimes E^*)=1$.
\end{definition}

\begin{Fact}
\begin{enumerate}
\renewcommand{\theenumi}{\alph{enumi}}
\item A simple vector bundle is necessarily indecomposable.
\item If a vector bundle is stable then it is simple.
\end{enumerate}
\end{Fact}

\section{Monads and bundles on $X = (\PP^{n_1})^2\times\ldots\times(\PP^{n_s})^2$ }

\noindent We now set up for monads on the multiprojective space $X = (\PP^{n_1})^2\times(\PP^{n_2})^2\times\cdots\times(\PP^{n_s})^2$.

\begin{lemma}
Let $\alpha_1,\cdots,\alpha_s$ and $k$ be positive integers, given $s$ matrices $f_{\alpha_1},\cdots,f_{\alpha_s}$ as shown;
\vspace{0.5cm}

\[ f_{\alpha_i} =\left[ \begin{array}{ccccc|ccccc}
&y^{\alpha_i}_{n_i} \cdots y^{\alpha_i+a_ik}_{0} & & -x^{\alpha_i}_{n_i} \cdots  -x^{\alpha_i+a_ik}_{0}\\
\adots&\adots &\adots&\adots \\
y^{\alpha_i}_{n_i} \cdots y^{\alpha_i+a_ik}_{0} & & -x^{\alpha_i}_{n_i} \cdots  -x^{\alpha_i+a_ik}_{0} \end{array} \right]_{k\times 2{(n_i+k)}}\]

where $i=1,\cdots,s$
\vspace{0.5cm}

\[ g_{\alpha_i} =\left[\begin{array}{cccccc}
x^{\alpha_i}_{0}\\
\vdots &\ddots
 & x^{\alpha_i+a_ik}_{0}\\
x^{\alpha_i}_{n_i} &\ddots &\vdots\\
&& x^{\alpha_i+a_ik}_{n_i}\\
y^{\alpha_i}_{0}\\\\
\vdots &\ddots
 & y^{\alpha_i+a_ik}_{0}\\
y^{\alpha_i}_{n_i} &\ddots &\vdots\\
&& y^{\alpha_i+a_ik}_{n_i}
\end{array} \right]_{{2(n_i+k)}\times k}\] 

\vspace{0.5cm}

for non-negative integers $a_1,\cdots,a_s$  we define two matrices $f$ and $g$ as follows\\
\[ f =\left[\begin{array}{cccc}
f_{\alpha_1}  &  f_{\alpha_2} \cdots & f_{\alpha_s}\end{array} \right]\] and

\[ g =\left[\begin{array}{cc}g_{\alpha_1} \\ g_{\alpha_2} \\\vdots\\g_{\alpha_s}\end{array} \right].\]

then :\\
\begin{enumerate}
   \item $f\cdot g = 0$ and
  \item The matrices $f$ and $g$ have maximal rank
\end{enumerate}
\end{lemma}

\begin{proof}

\begin{enumerate}
   \item Now    \[ f\cdot g =\left[\begin{array}{cccc}f_{\alpha_1}g_{\alpha_1}  &  f_{\alpha_2}g_{\alpha_2} \cdots & f_{\alpha_s}g_{\alpha_s}\end{array} \right]\]
   Since we have
  $\displaystyle{f_{\alpha_i}\cdot g_{\alpha_i}=\sum_{i=0}^s\sum_{j=0}^{n_s}\left(x^{\alpha_i}_jy^{\alpha_i}_j-y^{\alpha_i}_jx^{\alpha_i}_j\right)}$ and\\
 then it follows $f\cdot g $ is the zero matrix.
\\
 \item Notice that the rank of the two matrices drops if and only if all 
 $x^{\alpha_i}_{0},\cdots,x^{\alpha}_{n_i}$, $y^{\alpha_i}_{0},\cdots,v^{\alpha_i}_{n_i}$,
 for all $i=1,\cdots,s$ are zeros and this is not possible in a projective space. Hence maximal rank.
\end{enumerate}
\end{proof}

\noindent Using the matrices given in the above lemma we are going to construct a monad.

\begin{theorem}
Let $\alpha_1,\cdots,\alpha_s$ and $k$ be nonnegative integers. Then there exists a linear monad on $X = (\PP^{n_1})^2\times(\PP^{n_2})^2\times\cdots\times(\PP^{n_s})^2$ of the form;
\[\begin{CD}0\rightarrow{\OO_X(-\alpha_1,-\alpha_1,\cdots,-\alpha_s,-\alpha_s)^{\oplus k}} @>>^{f}>{\mathscr{G}_{\alpha_1}\oplus\cdots\oplus\mathscr{G}_{\alpha_s}}@>>^{g}>\OO_X(\alpha_1,\alpha_1,\cdots,\alpha_s,\alpha_s)^{\oplus k}\rightarrow0\end{CD}\]
where 
\begin{align*}\mathscr{G}_{\alpha_1}:=\OO_X(-\alpha_1,0,0,\cdots,0)^{\oplus n_1+\oplus k}\oplus\OO_X(0,-\alpha_1,0,0,\cdots,0)^{\oplus n_1+\oplus k}\\
\mathscr{G}_{\alpha_2}:=\OO_X(0,0,-\alpha_2,\cdots,0)^{\oplus n_2+\oplus k}\oplus\OO_X(0,0,0,-\alpha_2,\cdots,0)^{\oplus n_2+\oplus k}\\
\cdots\cdots\cdots\cdots\cdots\cdots\cdots\cdots\cdots\cdots\cdots\cdots\cdots\cdots\cdots\cdots\cdots\cdots\cdots\\
\mathscr{G}_{\alpha_s}:=\OO_X(0,0,\cdots,0,-\alpha_s,0)^{\oplus n_s+\oplus k}\oplus\OO_X(0,0,\cdots,0,-\alpha_s)^{\oplus n_s+\oplus k}
\end{align*}

\end{theorem}

\begin{proof}
The maps $f$ and $g$ in the monad are the matrices given in Lemma 3.1.\\
Notice that\\
$f\in$ Hom$(\OO_X(-\alpha_1,-\alpha_1,\cdots,-\alpha_s,-\alpha_s)^{\oplus k},\mathscr{G}_{\alpha_1}\oplus\cdots\oplus\mathscr{G}_{\alpha_s})$ and \\
$g\in$ Hom$(\mathscr{G}_{\alpha_1}\oplus\cdots\oplus\mathscr{G}_{\alpha_s},\OO_X(\alpha_1,\alpha_1,\cdots,\alpha_s,\alpha_s)^{\oplus k})$. \\
Hence by the above lemma they define the desired monad.
\end{proof}

\begin{lemma}
Let $K$ be the kernel bundle that sits in the short exact sequence
\[\begin{CD}0@>>>K @>>>\mathscr{G}_{\alpha_1}\oplus\cdots\oplus\mathscr{G}_{\alpha_s}@>>^{g}>\OO_X(\alpha_1,\alpha_1,\cdots,\alpha_s,\alpha_s)^{\oplus k} @>>>0\end{CD}\]
where 
\begin{align*}
\mathscr{G}_{\alpha_1}:=\OO_X(-\alpha_1,0,0,\cdots,0)^{\oplus n_1+\oplus k}\oplus\OO_X(0,-\alpha_1,0,0,\cdots,0)^{\oplus n_1+\oplus k}\\
\mathscr{G}_{\alpha_2}:=\OO_X(0,0,-\alpha_2,\cdots,0)^{\oplus n_2+\oplus k}\oplus\OO_X(0,0,0,-\alpha_2,\cdots,0)^{\oplus n_2+\oplus k}\\
\cdots\cdots\cdots\cdots\cdots\cdots\cdots\cdots\cdots\cdots\cdots\cdots\cdots\cdots\cdots\cdots\cdots\cdots\cdots\\
\mathscr{G}_{\alpha_s}:=\OO_X(0,0,\cdots,0,-\alpha_s,0)^{\oplus n_s+\oplus k}\oplus\OO_X(0,0,\cdots,0,-\alpha_s)^{\oplus n_s+\oplus k}
\end{align*}
\end{lemma}

\begin{proof}

We need to show that $H^0(X,\bigwedge^q K(p_{\alpha_{11}},p_{\alpha_{12}},p_{\alpha_{21}},p_{\alpha_{22}},\ldots,p_{\alpha_{s1}},p_{\alpha_{s2}}))=0$ 
for all $p_{\alpha_{11}}+p_{\alpha_{12}}+\cdots+p_{\alpha_{s1}}+p_{\alpha_{s2}}<0$ and $1\leq q\leq \rk(K)-1$.\\
\\
Consider the ample line bundle $\mathscr{L} = \OO_X(\alpha_1,\alpha_1,\ldots,\alpha_s,\alpha_s) = \OO(L)$. \\
Its class in 
$\pic(X)= \left\langle g_{\alpha_{11}},g_{\alpha_{12}},g_{\alpha_{21}},g_{\alpha_{22}},\ldots,g_{\alpha_{s1}},g_{\alpha_{s2}}\right\rangle$ corresponds to the class\\
$1\cdot[g_{\alpha_{11}}\times\PP^{n_1}]+1\cdot[\PP^{n_1}\times g_{\alpha_{12}}]+1\cdot[g_{\alpha_{21}}\times\PP^{n_2}]+1\cdot[\PP^{n_2}\times g_{\alpha_{22}}]+\cdots+1\cdot[g_{\alpha_{s1}}\times\PP^{n_s}]+1\cdot[\PP^{n_s}\times g_{\alpha_{s2}}]$ and
\\
Now from the display diagram of the monad we get \\ 
$c_1(T) = c_1(\mathscr{G}_{\alpha_1}\oplus\cdots\oplus\mathscr{G}_{\alpha_s}) - c_1(\OO_X(\alpha_1,\alpha_1,\ldots,\alpha_s,\alpha_s)^{\oplus k})\\
       = (n_1+k)[(-\alpha_1,0,0,\cdots,0)+(0,-\alpha_1,0,0,\cdots,0)] + \\
       (n_2+k)[(0,0,-\alpha_2,\cdots,0)+(0,0,0,-\alpha_2,\cdots,0)]+\cdots +\\
       (n_s+k)[(0,0,\cdots,0,-\alpha_s,0)+(0,0,\cdots,0,0,-\alpha_s)]- k(\alpha_1,\alpha_1,\cdots,\alpha_s,\alpha_s) \\
       = (-n_1\alpha-2k\alpha_1,-n_1\alpha_1-2k\alpha_1,\cdots,-n_s\alpha_s-2k\alpha_s,-n_s\alpha_s-2k\alpha_s) $.\\
Since $\displaystyle{L^{2(n_1+\dots+n_s)}>0}$ the degree of $K$ is $\deg_{\mathscr{L}}K = c_1(K)\cdot\mathscr{L}^{d-1}$\\
\begin{align*}
\begin{split}
\deg_{\mathscr{L}}T=-(n_1+\cdots+n_s+2sk)([g_1\times\PP^{n_1}]+[\PP^{n_1}\times g_2]+\cdots+[g_{\alpha_s}\times\PP^{n_s}]+[\PP^{n_s}\times g_{\alpha_s}])\\
(1\cdot[g_1\times\PP^{n_1}]+1\cdot[\PP^{n_1}\times g_2]+\cdots+1\cdot[g_{\alpha_s}\times\PP^{n_s}]+1\cdot[\PP^{n_s}\times g_{\alpha_s}])^{2\sum_{i=1}^s{n_i}-1}\\
\end{split}
\end{align*}
$\displaystyle{=-({\sum_{i=1}^sn_i}+2sk)L^{2{\sum_{i=1}^sn_i}}< 0}$\\
\\
Since $\deg_{\mathscr{L}}K<0$, then $(\bigwedge^q K)_{\mathscr{L}-norm} = (\bigwedge^q K)$ and  it suffices by 
Lemma 2.5, to prove that $h^0(\bigwedge^q K(p_{\alpha_{11}},p_{\alpha_{12}},p_{\alpha_{21}},p_{\alpha_{22}},\ldots,p_{\alpha_{s1}},p_{\alpha_{s2}})) = 0$ with $p_{\alpha_{11}}+p_{\alpha_{12}}+\cdots+p_{\alpha_{s1}}+p_{\alpha_{s2}}<0$ and $1\leq q\leq \rk(K)-1$.\\
\\
First, twist the exact sequence 
\[\begin{CD}0@>>>K @>>>\mathscr{G}_{\alpha_1}\oplus\cdots\oplus\mathscr{G}_{\alpha_s}@>>^{g}>\OO_X(\alpha_1,\cdots,\alpha_s)^{\oplus k} @>>>0\end{CD}\]
by $\OO_X(p_{\alpha_{11}},p_{\alpha_{12}},p_{\alpha_{21}},p_{\alpha_{22}},\ldots,p_{\alpha_{s1}},p_{\alpha_{s2}})$ we get,
\[
0\lra K(p_{\alpha_{11}},p_{\alpha_{12}},\ldots,p_{\alpha_{s1}},p_{\alpha_{s2}})\lra\mathscr{\overline{G}}_{\alpha_1}\oplus\cdots\oplus\mathscr{\overline{G}}_{\alpha_S}\lra\OO_X(\alpha_1+p_{\alpha_{11}},\alpha_2+p_{\alpha_{12}},\ldots,\alpha_s+p_{\alpha_{s2}})^{\oplus k}\lra0\]
where 
\begin{align*}
\mathscr{\overline{G}}_{\alpha_1}:=\OO_X(p_{\alpha_{11}}-\alpha_1,p_{\alpha_{12}},p_{\alpha_{21}},p_{\alpha_{22}}\ldots,p_{\alpha_{s1}},p_{\alpha_{s2}})^{\oplus n_1+\oplus k}\oplus\OO_X(p_{\alpha_{11}},p_{\alpha_{12}}-\alpha_1,p_{\alpha_{21}},\ldots,p_{\alpha_{s1}},p_{\alpha_{s2}})^{\oplus n_1+\oplus k}\\
\mathscr{\overline{G}}_{\alpha_2}:=\OO_X(p_{\alpha_{11}},p_{\alpha_{12}},p_{\alpha_{21}}-\alpha_2,\ldots,p_{\alpha_{s1}},p_{\alpha_{s2}})^{\oplus n_2+\oplus k}\oplus\OO_X(p_{\alpha_{11}},p_{\alpha_{12}},p_{\alpha_{21}},p_{\alpha_{22}}-\alpha_2,\ldots,p_{\alpha_{s1}},p_{\alpha_{s2}})^{\oplus n_2+\oplus k}\\
\cdots\cdots\cdots\cdots\cdots\cdots\cdots\cdots\cdots\cdots\cdots\cdots\cdots\cdots\cdots\cdots\cdots\cdots\cdots\cdots\cdots\cdots\cdots\cdots\cdots\cdots\cdots\cdots\cdots\cdots\cdots\\
\mathscr{\overline{G}}_{\alpha_s}:=\OO_X(p_{\alpha_{11}},p_{\alpha_{12}},p_{\alpha_{21}},\ldots,p_{\alpha_{s1}}-\alpha_s,p_{\alpha_{s2}})^{\oplus n_s+\oplus k}\oplus\OO_X(p_{\alpha_{11}},p_{\alpha_{12}},p_{\alpha_{21}},p_{\alpha_{22}},\ldots,p_{\alpha_{s1}},p_{\alpha_{s2}}-\alpha_s)^{\oplus n_s+\oplus k}
\end{align*}

and taking the exterior powers of the sequence by Fact 2.6 we get
\[0\lra \bigwedge^q K(p_{\alpha_{11}},p_{\alpha_{12}},p_{\alpha_{21}},p_{\alpha_{22}},\ldots,p_{\alpha_{s1}},p_{\alpha_{s2}}) \lra \bigwedge^q (\mathscr{\overline{G}}_{\alpha_1}\oplus\cdots\oplus\mathscr{\overline{G}}_{\alpha_s})\lra \cdots\]
Taking cohomology we have the injection:
\[0\lra H^0(X,\bigwedge^{q}K(p_{\alpha_{11}},p_{\alpha_{12}},p_{\alpha_{21}},p_{\alpha_{22}},\ldots,p_{\alpha_{s1}},p_{\alpha_{s2}}))\hookrightarrow H^0(X,\bigwedge^q (\mathscr{\overline{G}}_{\alpha_1}\oplus\cdots\oplus\mathscr{\overline{G}}_{\alpha_s}))\]
From Theorem 2.7 and Lemma 2.9 we have $H^0(X,\bigwedge^q (\mathscr{\overline{G}}_{\alpha_1}\oplus\cdots\oplus\mathscr{\overline{G}}_{\alpha_s}))=0$.\\
\\
$\Longrightarrow$ $H^0(X,\bigwedge^{q}K(p_{\alpha_{11}},p_{\alpha_{12}},p_{\alpha_{21}},p_{\alpha_{22}},\ldots,p_{\alpha_{s1}},p_{\alpha_{s2}})) =  H^0(X,\bigwedge^q (\mathscr{\overline{G}}_{\alpha_1}\oplus\cdots\oplus\mathscr{\overline{G}}_{\alpha_s})=0$ hence $K$ is stable.

\end{proof}

\begin{lemma} The cohomology bundle $E$ associated to the monad
\[\begin{CD}0\rightarrow{\OO_X(-\alpha_1,-\alpha_1,\cdots,-\alpha_s,-\alpha_s)^{\oplus k}} @>>^{f}>{\mathscr{G}_{\alpha_1}\oplus\cdots\oplus\mathscr{G}_{\alpha_s}}@>>^{g}>\OO_X(\alpha_1,\alpha_1,\cdots,\alpha_s,\alpha_s)^{\oplus k} \rightarrow0\end{CD}\]
of rank $\displaystyle{2(n_1+\dots+n_s)+2k(s-1)}$ is simple where $X = (\PP^{n_1})^2\times\cdots\times(\PP^{n_s})^2$.
\end{lemma}

\begin{proof}
The display of the monad is
\[
\begin{CD}
@.0@.0\\
@.@VVV@VVV\\
0\rightarrow{\OO_X(-\alpha_1,-\alpha_1,\cdots,-\alpha_s,-\alpha_s)^{\oplus k}} @>>>K=\ker g@>>>E\rightarrow0\\
||@.@VVV@VVV\\
0\rightarrow{\OO_X(-\alpha_1,-\alpha_1,\cdots,-\alpha_s,-\alpha_s)^{\oplus k}} @>>^{f}>{\mathscr{G}_{\alpha_1}\oplus\cdots\oplus\mathscr{G}_{\alpha_s}}@>>>Q=\cok f\rightarrow0\\
@.@V^{g}VV@VVV\\
@.{\OO_X(\alpha_1,\alpha_1,\cdots,\alpha_s,\alpha_s)^{\oplus k}}@={\OO_X(\alpha_1,\alpha_1,\cdots,\alpha_s,\alpha_s)^{\oplus k}}\\
@.@VVV@VVV\\
@.0@.0
\end{CD}
\]

\noindent Since $K$ the kernel of the map $g$ is stable from the above Lemma 3.3, we prove that the cohomology bundle $E=\ker g/\im f$  is simple.\\
\\
Take the dual of the short exact sequence 
\[\begin{CD}
0@>>>\OO_X(-\alpha_1,-\alpha_1,\cdots,-\alpha_s,-\alpha_s)^{\oplus k} @>>>K@>>>E @>>>0
\end{CD}\]
to get
\[
\begin{CD}
0@>>>E^* @>>>K^* @>>>\OO_X(\alpha_1,\alpha_1,\cdots,\alpha_s,\alpha_s)^{\oplus k}@>>>0.
\end{CD}
\]
Tensoring by $E$ we get\\
\[
\begin{CD}
0@>>>E\otimes E^* @>>>E\otimes K^* @>>>E(\alpha_1,\alpha_1,\cdots,\alpha_s,\alpha_s)^k@>>>0.
\end{CD}
\]
Now taking cohomology yields:
\[\begin{CD}
0@>>>H^0(X,E\otimes E^*) @>>>H^0(X,E\otimes K^*) @>>>H^0(E(\alpha_1,\alpha_1,\cdots,\alpha_s,\alpha_s)^{\oplus k})@>>>\cdots
\end{CD}\]
\\
which implies that 
\begin{equation}
h^0(X,E\otimes E^*) \leq h^0(X,E\otimes K^*)
\end{equation}
\\
Now we dualize the short exact sequence
\[\begin{CD}
0@>>>K@>>>\mathscr{G}_{\alpha_1}\oplus\cdots\oplus\mathscr{G}_{\alpha_s} @>>>\OO_X(\alpha_1,\alpha_1,\cdots,\alpha_s,\alpha_s)^{\oplus k} @>>>0
\end{CD}\]
\\
to get
\[\begin{CD}
0@>>>\OO_X(-\alpha_1,-\alpha_1,\cdots,-\alpha_s,-\alpha_s)^{\oplus k} @>>>{\mathscr{G}_{\alpha_1}\oplus\cdots\oplus\mathscr{G}_{\alpha_s}} @>>>K^* @>>>0
\end{CD}\]
\\
Twisting the short exact sequence above by $\OO_X(-\alpha_1,-\alpha_1,\cdots,-\alpha_s,-\alpha_s)$ yields
\[\begin{CD}
0@>>>\OO_X(-2\alpha_1,-2\alpha_1,-2\alpha_2,-2\alpha_2,\cdots,-2\alpha_s,-2\alpha_s)^{\oplus k} @>>>{\mathscr{G}'_{\alpha_1}\oplus\cdots\oplus\mathscr{G}'_{\alpha_s} }\\ @>>>K^*(-\alpha_1,-\alpha_1,\cdots,-\alpha_s,-\alpha_s) @>>>0
\end{CD}\]
where 
\begin{align*}
\mathscr{G}'_{\alpha_1}:=\OO_X(-2\alpha_1,-\alpha_1,-\alpha_2,-\alpha_2,\cdots,-\alpha_s,-\alpha_s)^{\oplus{n_1}+\oplus k}\oplus\OO_X(-\alpha_1,-2\alpha_1,-\alpha_2,-\alpha_2,\cdots,-\alpha_s,-\alpha_s)^{\oplus{n_1}+\oplus k}\\
\mathscr{G}'_{\alpha_2}:=\OO_X(-\alpha_1,-\alpha_1,-2\alpha_2,-\alpha_2,\cdots,-\alpha_s,-\alpha_s)^{\oplus{n_2}+\oplus k}\oplus\OO_X(-\alpha_1,-\alpha_1,-\alpha_2,-2\alpha_2,\cdots,-\alpha_s,-\alpha_s)^{\oplus{n_2}+\oplus k}\\
\cdots\cdots\cdots\cdots\cdots\cdots\cdots\cdots\cdots\cdots\cdots\cdots\cdots\cdots\cdots\cdots\cdots\cdots\cdots\cdots\cdots\cdots\cdots\cdots\cdots\cdots\cdots\cdots\cdots\cdots\cdots\\
\mathscr{G}'_{\alpha_s}:=\OO_X(-\alpha_1,-\alpha_1,-\alpha_2,-\alpha_2,\cdots,-2\alpha_s,-\alpha_s)^{\oplus{n_s}+\oplus k}\oplus\OO_X(-\alpha_1,-\alpha_1,-\alpha_2,-\alpha_2,\cdots,-\alpha_s,-2\alpha_s)^{\oplus{n_s}+\oplus k}
\end{align*}%
next on taking cohomology one gets
\[\begin{CD}
0\lra H^0(\OO_X(-2\alpha_1,-2\alpha_1,-2\alpha_2,-2\alpha_2,\cdots,-2\alpha_s,-2\alpha_s)^k) \lra H^0(\mathscr{G'}_{\alpha_1})\oplus \cdots\oplus H^0(\mathscr{G'}_{\alpha_s})\lra\\\lra H^0(K^*(-\alpha_1,-\alpha_1,\cdots,-\alpha_s,-\alpha_s) )\lra\\
\lra H^1(\OO_X(-2\alpha_1,-2\alpha_1,-2\alpha_2,-2\alpha_2,\cdots,-2\alpha_s,-2\alpha_s)^k) \lra H^1(\mathscr{G'}_{\alpha_1})\oplus \cdots\oplus H^1(\mathscr{G'}_{\alpha_s})\lra\\\lra H^1(K^*(-\alpha_1,-\alpha_1,\cdots,-\alpha_s,-\alpha_s))\lra\\
\lra H^2(X,\OO_X(-2\alpha_1,-2\alpha_1,-2\alpha_2,-2\alpha_2,\cdots,-2\alpha_s,-2\alpha_s)^k) \lra H^2(\mathscr{G'}_{\alpha_1})\oplus \cdots\oplus H^2(\mathscr{G'}_{\alpha_s})\lra\\\lra H^2(K^*(-\alpha_1,-\alpha_1,\cdots,-\alpha_s,-\alpha_s))\lra\cdots
\end{CD}
\]
As a consequence of from Theorem 2.7 and Lemma 2.8 we deduce  that \\
$H^0(X,K^*(-\alpha_1,-\alpha_1,\cdots,-\alpha_s,-\alpha_s)) = 0$ and $H^1(X,K^*(-\alpha_1,-\alpha_1,\cdots,-\alpha_s,-\alpha_s)) = 0$ 
\\
Lastly, tensor the short exact sequence
\[
\begin{CD}
0@>>>\OO(-\alpha_1,-\alpha_1,\cdots,-\alpha_s,-\alpha_s)^{\oplus k} @>>>K @>>> E@>>>0\\
\end{CD}
\]
by $K^*$ to get
\[
\begin{CD}
0@>>>K^*(-\alpha_1,-\alpha_1,\cdots,-\alpha_s,-\alpha_s)^k @>>>K\otimes K^* @>>> E\otimes K^*@>>>0\\
\end{CD}
\]
and taking cohomology we have
\\
\[
\begin{CD}
0@>>>H^0(X,K^*(-\alpha_1,-\alpha_1,\cdots,-\alpha_s,-\alpha_s)^k) @>>>H^0(X,K\otimes K^*) @>>> H^0(X,E\otimes K^*)@>>>\\
@>>>H^1(X,K^*(-\alpha_1,-\alpha_1,\cdots,-\alpha_s,-\alpha_s)^k)@>>>\cdots
\end{CD}
\]
\\
But since  $H^0(X,K^*(-\alpha_1,-\alpha_1,\cdots,-\alpha_s,-\alpha_s)) = H^1(X,K^*(-\alpha_1,-\alpha_1,\cdots,-\alpha_s,-\alpha_s)) = 0$ from above 
then  it follows $H^1(X,K^*(-\alpha_1,-\alpha_1,\cdots,-\alpha_s,-\alpha_s)^k)=0$ for $k>1$.\\
\\
so we have 
\\
\[
\begin{CD}
0@>>>H^0(X,K^*(-\alpha_1,-\alpha_1,\cdots,-\alpha_s,-\alpha_s)^{k}) @>>>H^0(X,K\otimes K^*) @>>> H^0(X,E\otimes K^*)@>>>0
\end{CD}
\]
\\
This implies that 
\begin{equation}
h^0(X,K\otimes K^*) \leq h^0(X,E\otimes K^*)
\end{equation}
\\
Since $K$ is stable then it follows that it is simple which implies $h^0(X,K\otimes K^*)=1$.\\
\\
From $(3)$ and $(4)$ and putting these together we have;\\
\[1\leq h^0(X,E\otimes E^*) \leq h^0(X,E\otimes K^*) = h^0(X,K\otimes K^*) = 1\]\\
\\
We have $ h^0(X,E\otimes E^*) = 1 $ and therefore $E$ is simple.

\end{proof}

\begin{theorem}
Let $X = (\PP^{n_1})^2\times(\PP^{n_2})^2\times\cdots\times(\PP^{n_s})^2$ and $\mathscr{L}=\OO_X(\alpha_1,\alpha_1,\cdots,\alpha_s,\alpha_s)$ and ample line bundle,
then the monad 
\[\begin{CD}0\rightarrow{\OO_X(-\alpha_1,-\alpha_1,\cdots,-\alpha_s,-\alpha_s)^{\oplus k}} @>>^{f}>{\mathscr{G}_{\alpha_1}\oplus\cdots\oplus\mathscr{G}_{\alpha_s}}@>>^{g}>\OO_X(\alpha_1,\alpha_1,\cdots,\alpha_s,\alpha_s)^{\oplus k}\rightarrow0\end{CD}\]
where 
\begin{align*}\mathscr{G}_{\alpha_1}:=\OO_X(-\alpha_1,0,0,\cdots,0)^{\oplus n_1+\oplus k}\oplus\OO_X(0,-\alpha_1,0,0,\cdots,0)^{\oplus n_1+\oplus k}\\
\mathscr{G}_{\alpha_2}:=\OO_X(0,0,-\alpha_2,\cdots,0)^{\oplus n_2+\oplus k}\oplus\OO_X(0,0,0,-\alpha_2,\cdots,0)^{\oplus n_2+\oplus k}\\
\cdots\cdots\cdots\cdots\cdots\cdots\cdots\cdots\cdots\cdots\cdots\cdots\cdots\cdots\cdots\cdots\cdots\cdots\cdots\\
\mathscr{G}_{\alpha_s}:=\OO_X(0,0,\cdots,0,-\alpha_s,0)^{\oplus n_s+\oplus k}\oplus\OO_X(0,0,\cdots,0,-\alpha_s)^{\oplus n_s+\oplus k}
\end{align*}
has the properties:
\begin{enumerate}\renewcommand{\theenumi}{\alph{enumi}}
 \item The kernel bundle, $T=\ker(g)$ is $\mathscr{L}$-stable.
 
 \item The cohomology bundle $E$ of $\rk(E)=2(n_1+\dots+n_s+k(s-1))$ is simple.

\end{enumerate}

\end{theorem}

\begin{proof}
(a) Follows from Lemma 4.3 and (b) follows from Lemma 4.4.
\end{proof}

\vspace{1cm}

\noindent Having established monads on different spaces see \cite{4,5,6,7,8,9,10}, we did not answer the question of how many such monads exist.
The set of pairs of morphisms that define a monad on an algebraic variety have an algebraic structure. We shall study the moduli problems for these spaces
in the future.

\noindent \textbf{Data Availability statement}
My manuscript has no associate data.

\vspace{1cm}

\noindent \textbf{Conflict of interest}
The author states that there is no conflict of interest.


\newpage


\begin{thebibliography}{1}
\bibitem[1]{1} Ancona V and Ottaviani G: \textit{Stability of special instanton Bundles on $\PP^{2n+1}$}.
Transactions of the American Mathematical Society 341 (1994) 677 - 693. 
\href{https://doi.org/10.2307/2154578}{\color{blue}doi: 10.2307/2154578.}
%
\bibitem[2]{2} Fl\o{}ystad G. Monads on a Projective Space. Communications in Algebra, 28 (2000), 5503 - 5516.
\href{https://www.tandfonline.com/doi/abs/10.1080/00927870008827171}{\color{blue}doi: 10.1080/00927870008827171.}
%
\bibitem[3]{3} Jardim M, Menet M, Prata D and Earp HNS\'{a}. Holomorphic bundles for higher dimensional gauge theory.
Bulletin London Mathematical society, 49 (2017). \href{https://doi.org/10.1112/blms.12017}{\color{blue}doi: 10.1112/blms.12017.}
%

\bibitem[4]{4} Maingi D. Vector Bundles of low rank on a multiprojective space. Le Matematiche. Vol. LXIX (2014) - Fasc. II. pp 31-41.
\href{https://lematematiche.dmi.unict.it/index.php/lematematiche/article/view/1048}{\color{blue}doi: 10.4418/2014.69.2.4.}
%
\bibitem[5]{5} Maingi D (2021). Indecomposable Vector Bundles associated to Monads on Cartesian products of projective spaces.
Turkish Journal of Mathematics. Vol. 45: No. 5. Article 17. Pages 2126-2139.
\href{https://doi.org/10.3906/mat-2101-6}{\color{blue}doi: 10.3906/mat-2101-6}
%
\bibitem[6]{6} Maingi D. Monads on multiprojective Products of Projective Spaces.
Manuscripta Mathematica (2022),\href{https://doi.org/10.1007/s00229-022-01449-0}{\color{blue}doi: 10.1007/s00229-022-01449-0}.
%
\bibitem[7]{7} Maingi D. Vector Bundles associated to monads on Cartesian Products of Projective Spaces.
Open Journal of Mathematical Sciences, OMS - Vol 7 (2023), Issue 1, pp 148-159\href{https://doi.org/10.30538/oms2023.0203}{\color{blue}doi: 10.30538/oms2023-0203}.
%

\bibitem[8]{8} Maingi D (2024). Monads on Cartesian products of projective spaces. 
Commun. Korean Math. Soc. 2024;39:947-963.  \href{https://doi.org/10.4134/CKMS.c230316}{\color{blue}doi: doi.org/10.4134/CKMS.c230316}.
%

\bibitem[9]{9} Maingi D (2025). Indecomposable bundles on Cartesian products of odd projective spaces
preprint,\href{https://arxiv.org/pdf/2504.10321}{\color{blue}arXiv:2504.10321}.
%

\bibitem[10]{10} Maingi D. Vector Bundle construction via monads on multiprojective Spaces.
Open Journal of Mathematical Sciences, OMS - Vol 9 (2025), Issue 1, pp 158-171\href{https://doi.org/10.30538/oms2025.0251}{\color{blue}doi: 10.30538/oms2025-0251}.
%

\bibitem[11]{11} Marchesi S, Marques P M and Soares H. Monads on a Projective Varieties. Pacific Journal of Mathematics, vol 296 (2018), no. 1, 155-180. 
\href{https://doi.org/10.2140/pjm.2018.296.155}{\color{blue}doi: 10.2140/pjm.2018.296.155.}
%

\bibitem[12]{12} Okonek C, Schneider M and Spindler H. Vector Bundles on Complex Projective Spaces.
Springer, 1980, \href{https://doi.org/10.1007/978-1-4757-1460-9}{\color{blue}doi.org/10.1007/978-1-4757-1460-9}
%
\bibitem[13]{13} Perrin D. G\'{e}om\'{e}trie alg\'{e}brique. Une introduction, (1995), EDP Sciences/CNRS \'{e}dition.
\href{https://www.decitre.fr/livres/geometrie-algebrique-9782729605636.html}{\color{blue}ISBN-2-7296-0563-0}
%


\end{thebibliography}
\end{document}